\documentclass{article}
\usepackage[margin=20mm]{geometry}

\usepackage{amsmath, amssymb}
\usepackage{graphicx, wrapfig} 
\usepackage{amsthm}
\usepackage{url}
\usepackage{color}
\usepackage{mathrsfs}
\usepackage{bm}

\theoremstyle{plain}
\newtheorem{theorem}{Theorem}

\theoremstyle{definition}
\newtheorem{definition}{Definition}
\newtheorem{example}{Example}

\newtheorem{proposition}{Proposition}
\newtheorem{lemma}{Lemma}
\newtheorem{corollary}{Corollary}

\newtheorem{remark}{Remark}

\usepackage{color}

\title{Cabled braids and their crossing matrices}
\author{Ayaka Shimizu\thanks{Center for Soft Matter Physics, Ochanomizu University, 2-1-1, Otsuka, Bunkyo-ku, Tokyo, 112-8610, Japan. Email: shimizu.ayaka@ocha.ac.jp, shimizu1984@gmail.com} 
and Yoshiro Yaguchi\thanks{Maebashi Institute of Technology, 460-1, Kamisadori, Maebashi, Gunma, 371-0816, Japan. Email: y.yaguchi@maebashi-it.ac.jp}}
\date{\today}

\begin{document}

\maketitle

\begin{abstract}
We study cabling operations on braids, and characterize the matrices that can be realized as crossing matrices of reducible pure braids. 
We also use cabling to construct enhanced conjugacy invariants of braids. 
\end{abstract}

\section{Introduction}
\label{section-intro}

The {\it crossing matrix} $C(b)$ of a braid $b$, introduced in \cite{Bu}, is a matrix-valued invariant (see Section \ref{section-CM-cabled} for definition). 
In \cite{Bu}, crossing matrices of braids and those of pure braids are completely characterized\footnote{At the moment, the characterization of crossing matrices of positive pure braids has not been completed. See \cite{YAY, AY-5} for recent progress.}. 
In this paper, we study crossing matrices of cabled braids and give a complete characterization for crossing matrices of reducible pure braids. 

In this section, $n$ is an integer with $n\geq 3$. 
We prove the following theorem in Section \ref{section-CM-cabled}.

\begin{theorem}
Let $N=\{1, 2, \dots, n\}$. 
An $n\times n$ matrix $A=[a_{ij}]$ is the crossing matrix of a reducible pure $n$-braid if and only if $A$ is a zero-diagonal symmetric matrix and 
there exists a partition $N_1\amalg N_2\amalg \dots \amalg N_m$ of $N$ with $1<m<n$ such that for any distinct $k, l\in \{1, 2, \dots ,m\}$ and any $(i,j),\ (i',j')\in N_k\times N_l$, $a_{ij}=a_{i'j'}$.
\label{thm-CM-red}
\end{theorem}
\medskip


\noindent For a subset $S$ of $N$, let $\overline{S}:=N\setminus S$. 
We have the following corollary. 

\medskip 
\begin{corollary}
If an $n \times n$ matrix $A=[a_{ij}]$ is the crossing matrix of a reducible pure $n$-braid, then there exists a subset $S$ of $N= \{ 1, 2, \dots , n \}$ with $1<|S|<n$ such that for any $k\in \overline{S}$, $a_{kl}=a_{kl'}$ for some distinct $l,l'\in S$.
\label{cor-red2}
\end{corollary}
\medskip

\noindent We provide a useful sufficient condition for pure braids to be pseudo-Anosov (see Example \ref{ex-anosov}). 

\medskip 
\begin{corollary} 
Let $A=[a_{ij}]$ be an $n \times n$ zero-diagonal symmetric matrix. 
If $A$ satisfies that for any subset $S$ of $N= \{ 1, 2, \dots , n \}$ with $1<|S|<n$, $a_{kl}\ne a_{kl'}$ for any $k\in \overline{S}$ and any distinct $l, l'\in S$, then $A$ is the crossing matrix of a pseudo-Anosov pure $n$-braid. 
In particular, if $a_{kl}\ne a_{kl'}$ for any $k\in N$ and any distinct $l, l'\in N$, then $A$ is the crossing matrix of a pseudo-Anosov pure $n$-braid.
\label{cor-anosov}
\end{corollary}
\medskip

Cabling has been used to enhance polynomial invariants of knots and links.
(See, for example, \cite{LL, MT, JM}. See also \cite{T}.)
In this paper, we also investigate how cabling can be used to enhance conjugacy invariants of braids. 
Let $g$ be a conjugacy invariant of braids. 
In Appendix \ref{section-conjugacy}, we define a multiset $F_{\mathbf{\beta}}(g;b)$ that is defined for an $n$-braid $b$, an $n$-tuple of braids, and a conjugacy invariant $g$. 
We prove the following in Appendix \ref{section-conjugacy}.

\medskip 
\begin{proposition}
Let $b$, $b'$ be pure $n$-braids. 
If $b$ and $b'$ are conjugate, then $F_{\mathbf{\beta}}(g;b)=F_{\mathbf{\beta}}(g;b')$ for any $n$-tuple of braids $\mathbf{\beta}$ and conjugacy invariant $g$. 
\label{prop-p-Fb}
\end{proposition}
\medskip

\noindent As shown in Examples \ref{ex-w1}, \ref{ex-w2}, and \ref{ex-np-W}, cabling can make conjugacy invariants more effective in certain cases. \\

The rest of the paper is organized as follows. 
In Section \ref{section-braid}, we recall braids and define a cabling operation on braids. 
In Section \ref{section-CM-cabled}, we review the crossing matrix of braids and prove Theorem \ref{thm-CM-red}, Corollaries \ref{cor-red2} and \ref{cor-anosov}. 
In Section \ref{section-periodic}, we investigate the crossing matrix of periodic braids. 
In Appendix \ref{section-conjugacy}, we explore conjugacy invariants and prove Proposition \ref{prop-p-Fb}.

\section{Braids}
\label{section-braid}

\subsection{Braids and braid diagrams}

An {\it $n$-braid} $b$ consists of $n$ strands in $\mathbb{R}^3$, with endpoints on two horizontal bars, such that each strand runs monotonically from the upper bar to the lower bar. 
Let $\mathcal{B}_n$ denote the set of $n$-braids. 
Let $\mathcal{B}$ denote the set of all braids. 
The strand of an $n$-braid whose upper endpoint is in the $i^{th}$position from the left is called the {\it $i^{th}$ strand} and denoted by $s_i$. \\

A {\it braid diagram} $B$ of a braid $b$ is a regular projection of $b$ on $\mathbb{R}^2$ as depicted in Figure \ref{fig-BD}. 
\begin{figure}[ht]
\centering
\includegraphics[width=1.5cm]{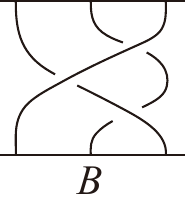}
\caption{A braid diagram $B$.}
\label{fig-BD}
\end{figure}
Each braid diagram can be represented by a word in the generators $\sigma_i$ and $\sigma_i^{-1}$, where $\sigma_i$ and $\sigma_i^{-1}$ are shown in Figure \ref{fig-s}. 
For example, the braid diagram $B$ shown in Figure \ref{fig-BD} is expressed as $\sigma_2 \sigma_1 \sigma_2^{-1}$. 
It is well known that two braid diagrams $B$ and $B'$ represent the same braid if and only if their word representations are related by a finite sequence of the following three types of transformations for $\varepsilon \in \{ +1, -1 \}$, $i \in \{ 1, 2, \dots, n-1 \}$: 
$\sigma_i^{\varepsilon} \sigma_i^{-\varepsilon}=1$, $\sigma_i \sigma_j \sigma_i = \sigma_j \sigma_i \sigma_j$ if $|i-j|=1$, and $\sigma_i \sigma_j = \sigma_j \sigma_i$ if $|j-i|>1$ (\cite{Artin-2}). 
\begin{figure}[ht]
\centering
\includegraphics[width=6cm]{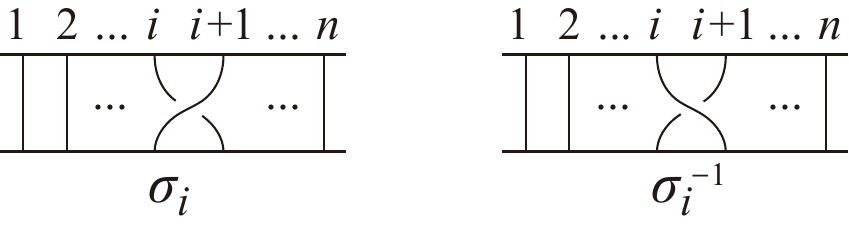}
\caption{$\sigma_i$ and $\sigma_i^{-1}$.}
\label{fig-s}
\end{figure}

Each braid $b \in \mathcal{B}_n$ induces a permutation on $\{ 1, 2, \dots , n \}$, defined as follows. 
Suppose that the $i^{th}$ strand has the lower endpoint in the $\pi (i)^{th}$ position from the left on the lower bar. 
The resulting permutation $\pi (1, 2, \dots , n)=(\pi(1), \pi(2), \dots , \pi(n))$ is called the {\it braid permutation} of $b$. 
For example, the braid illustrated in Figure \ref{fig-BD} has the braid permutation $\pi (1,2,3)=(3,2,1)$, and the order of $\pi$ is 2.

\subsection{Cabled braids}

Let $b$ be an $n$-braid with $n$ strands $s_1, s_2, \dots , s_n$. 
Let $\mathbf{\beta}=(b_1, b_2, \dots , b_n)$ be a tuple of braids $b_1, b_2, \dots , b_n \in \mathcal{B}$. 
A {\it cabled braid} $f(\mathbf{\beta}; b)=f(b_1, b_2, \dots, b_n;b)$ is a braid that is obtained from $b$ by replacing each strand $s_i$ with a braid $b_i$ in a sufficiently small tubular neighborhood of $s_i$ in ${\mathbb R}^3$ as shown in Figure \ref{fig-cabling}. 
\begin{figure}[ht]
\centering
\includegraphics[width=7cm]{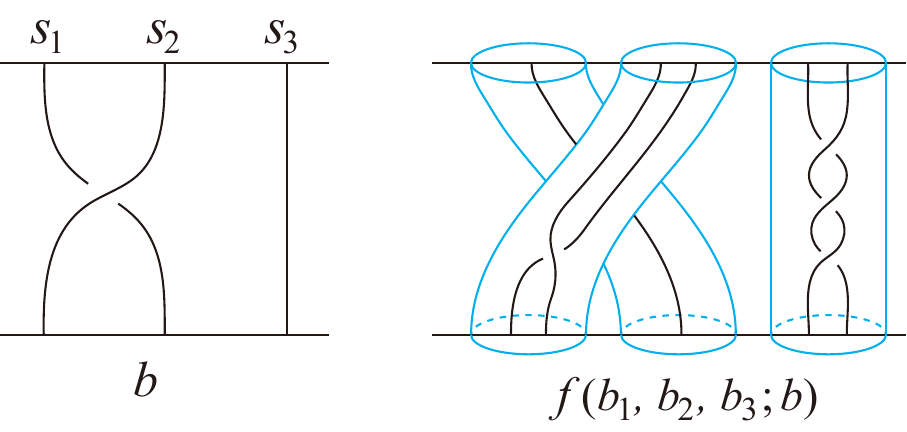}
\caption{A cabled braid $f(b_1, b_2, b_3;b)$ with $b_1=id_1 \in \mathcal{B}_1, b_2= \sigma_1^{-1} \in \mathcal{B}_2, b_3=\sigma_1^2 \in \mathcal{B}_2$. }
\label{fig-cabling}
\end{figure}
We call such a tubular neighborhood a {\it tube of $s_i$}. 
We denote the trivial $n$-braid by $id_n$. 
A braid $b$ is said to be {\it cabled} if it can be expressed as $b=f(b_1, b_2, \dots , b_m;a)$ with $m \geq 2$ and at least one of $b_1, b_2, \dots , b_m$ is not $id_1$. \\

Next, we define an integer-cabling. 
Let $k_1, k_2, \dots , k_n \in \mathbb{Z}_+=\{k\in Z\ |\ k>0\}$. 
We define a {\it $(k_1, k_2, \dots , k_n)$-cabling} of $b$ by $f(k_1, k_2, \dots , k_n;b)=f(id_{k_1}, id_{k_2}, \dots , id_{k_n} ;b)$. 
Note that $f(k_1, k_2, \dots , k_n ;b) \in \mathcal{B}_{k_1+k_2+\dots +k_n}$.

\subsection{Reducible braids}

According to the Nielsen–Thurston classification, every braid is periodic, reducible, or pseudo-Anosov (see, for example, \cite{GM}). In particular, every periodic pure $n$-braid is reducible for $n\geq 3$. 
Two braids $b, b' \in \mathcal{B}_n$ are {\it conjugate} and denoted by $b \stackrel{{ \mathrm{conj}}}{\backsim} b'$ if there exists a braid $a \in \mathcal{B}_n$ such that $b'=a^{-1}ba$. 
When $b'=a^{-1}ba$, we say that $a$ conjugates $b$ to $b'$. 
Let $\mathcal{P}_n$ be the set of all pure $n$-braids and let $\mathcal{P}$ be the set of all pure braids. 
In \cite{GM}, a necessary and sufficient condition for a pure braid to be reducible is shown. 

\medskip 
\begin{proposition}[\cite{GM}]
A pure braid $b \in \mathcal{P}_n$ is reducible if and only if $b$ is conjugate to a cabled braid $f(b_1, b_2, \dots , b_m;a)$, where $1<m<n$, $a \in \mathcal{P}_m$, $b_1, b_2, \dots , b_m \in \mathcal{P}$. 
\label{prop-GM}
\end{proposition}
\medskip 

\noindent We call such a cabled braid $f(b_1, b_2, \dots , b_m;a)$ a {\it standard form of a reducible pure braid}.

\section{Crossing matrix of cabled braids}
\label{section-CM-cabled}

\subsection{Crossing matrix of braids}

The {\it crossing matrix} $C(B)$ of an $n$-braid diagram $B$ is an $n \times n$ matrix $[m_{i j}]$ such that $m_{i i}=0$ for $i \in \{ 1, 2,\dots , n \}$ and $m_{i j}$ is the sum of the signs of the crossings where the $i^{th}$ strand is over the $j^{th}$ strand (\cite{Bu, Gu}). 
The matrix $C(b)$ does not depend on the choice of a diagram $B$ representing $b$. 
We therefore denote it by $C(b)$ and call it the {\it crossing matrix of $b$}. 
The following lemma is proved in \cite{Bu}. 

\medskip 
\begin{lemma}[\cite{Bu}]
An $n \times n$ zero-diagonal matrix $M$ is the crossing matrix of a pure $n$-braid if and only if $M$ is symmetric. 
\label{lem-CM-pure}
\end{lemma}
\medskip 

Let $M=[m_{ij}]$ be an $n \times n$ matrix and let $\rho$ be a permutation on $\{ 1, 2, \dots , n \}$. 
We define $\rho (M)=N=[n_{i j}]$ by setting $n_{i j}=m_{\rho^{-1}(i) \rho^{-1}(j)}$. 
It is shown in \cite{Bu} that $C(b_1 b_2)=C(b_1)+ \pi (C(b_2))$, where $\pi$ is the braid permutation of $b_1$. 
For the power of a braid, we have the following formula. 

\medskip 
\begin{proposition}[\cite{AS}]
Let $b$ be a braid with braid permutation $\pi$, and let $r$ be a non-negative integer. 
Then 
$$C(b^r)=\sum_{k=0}^{r-1} \pi^k (C(b)).$$
\label{prop-C-br}
\end{proposition}

\noindent For inverse braids, we have the following proposition. 

\medskip
\begin{proposition}
Let $b$ be an $n$-braid with braid permutation $\pi$. 
Let $C(b)=[m_{i j}]$, $C(b^{-1})=[ l_{i j}]$ be the crossing matrices of $b$, $b^{-1}$, respectively. 
Then $l_{i j}=- m_{\pi^{-1}(i) \ \pi^{-1}(j)}$.
\label{prop-CM-inverse}
\end{proposition}
\medskip 

\begin{proof}
Since $b^{-1}$ is obtained from $b$ by applying a horizontal rotation and then applying crossing changes at all crossings in a braid diagram, the $i^{th}$ strand of $b^{-1}$ corresponds to the $\pi^{-1}(i)^{th}$ strand of $b$. 
Since all corresponding signs are opposite between $b$ and $b^{-1}$, we have the relation $l_{i j}=- m_{\pi^{-1}(i) \ \pi^{-1}(j)}$. 
(See Figure \ref{fig-inverse}.)
\end{proof}
\medskip 

\begin{figure}[ht]
\centering
\includegraphics[width=11cm]{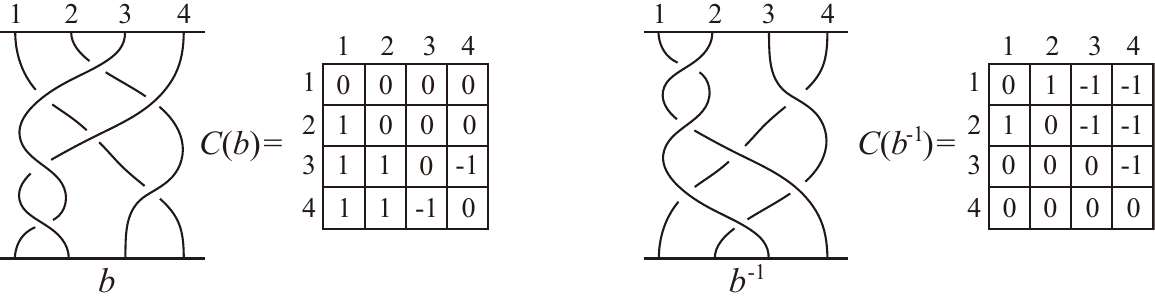}
\caption{A 4-braid $b$ with braid permutation $\pi(1,2,3,4)=(4,3,1,2)$ and its inverse. }
\label{fig-inverse}
\end{figure}

\noindent Two $n \times n$ matrices $A$ and $A'$ are {\it permutation equivalent} if $A'$ is obtained from $A$ by the same permutation to the rows and columns. 
For pure braids, we have the following. 

\medskip 
\begin{lemma}[\cite{AS}]
If two pure braids $b$ and $b'$ are conjugate, then the crossing matrices $C(b)$ and $C(b')$ are permutation equivalent.
\label{lem-CM-p-perm}
\end{lemma}
\medskip

\noindent Lemma \ref{lem-CM-p-perm} does not hold for non-pure braids. 
For general braids, we define the {\it purified} crossing matrix as follows. 

\medskip 
\begin{definition}
Let $b$ be a braid whose order of braid permutation is $r$. 
The {\it purified crossing matrix}, $C^{\text{pur}}(b)$,  is defined by $C^{\text{pur}}(b)=C(b^r)$. 
\end{definition}
\medskip 

\noindent Note that $C^{\text{pur}}(b)$ can be obtained from $C(b)$ and the braid permutation $\pi$ of $b$ by Proposition \ref{prop-C-br}. 
We have the following.

\medskip 
\begin{proposition}[\cite{AS}]
If two braids $b$ and $b'$ are conjugate, then the purified crossing matrices $C^{\text{pur}}(b)$ and $C^{\text{pur}}(b')$ are permutation equivalent. 
\label{prop-CM-pure-b}
\end{proposition}
\medskip 

\begin{proof}
If $b \stackrel{{ \mathrm{conj}}}{\backsim} b' \in \mathcal{B}_n$, then they have the same order $r$ of braid permutation and $b^r \stackrel{{ \mathrm{conj}}}{\backsim} (b')^r \in \mathcal{P}_n$. 
Hence, this proposition holds by Lemma \ref{lem-CM-p-perm}. 
\end{proof}
\medskip

\noindent To determine whether two given matrices are permutation equivalent or not, we generally need to examine up to $n!$ permutations. 
In \cite{AS, AY-H}, the multiset of entries, determinant, characteristic polynomial, etc., of $C^{\text{pur}}(b)$ are used as reasonable conjugacy invariants. 
(See Appendix \ref{section-conjugacy} for more details.)

\subsection{Crossing matrix of integer-cabled braids}
\label{subsection-integer}

For the crossing matrix of integer-cabled braids, we have the following proposition. 

\medskip
\begin{proposition}
Let $b \in \mathcal{B}_n$. 
Let $b^{(k)}=f(b_1, b_2, \dots , b_n;b) \in \mathcal{B}_{n+1}$ with $b_1= \dots =b_{k-1}=b_{k+1}= \dots = b_n=id_1$ and $b_k=id_2$ ($k \in \{ 1, 2, \dots , n \}$). 
Then $C(b^{(k)})$ is obtained from $C(b)$ as follows. 
Divide $C(b)$ into $3 \times 3$ blocks by separating $k-1, 1, n-k$ rows and columns as follows: 
\begin{align*}
C(b)= \begin{bmatrix}
A_{11} & A_{12} & A_{13} \\
A_{21} & 0 & A_{23} \\
A_{31} & A_{32} & A_{33} 
\end{bmatrix}.
\end{align*}
The crossing matrix of $b^{(k)}$ is given by 
\begin{align*}
C(b^{(k)})= \begin{bmatrix}
A_{11} & A_{12} & A_{12} & A_{13} \\
A_{21} & 0 & 0 & A_{23} \\
A_{21} & 0 & 0 & A_{23} \\
A_{31} & A_{32} & A_{32} & A_{33} 
\end{bmatrix}.
\end{align*}
Namely, $C(b^{(k)})$ is obtained from $C(b)$ by duplicating the $k^{th}$ row and column. 
When $k=1$ or $n$, this is obtained from a $2 \times 2$ block decomposition. 
\label{prop-CM-2}
\end{proposition}
\medskip 

\noindent Iterating the construction of Proposition \ref{prop-CM-2} yields the crossing matrix of $f(k_1, k_2, \dots , k_n;b)$ for any $k_1, k_2, \dots , k_n \in \mathbb{Z}_+$. 

\medskip 
\begin{example}
For the 4-braid $b$ in Figure \ref{fig-cb}, the crossing matrix of $f(1,2,1,1;b)$ is obtained by duplicating the $2^{nd}$ row and column of $C(b)$. 
Then the crossing matrix of $f(1,2,1,2;b)$ is obtained by duplicating the $5^{th}$ row and column of $C(f(1,2,1,1;b))$. 
\begin{figure}[ht]
\centering
\includegraphics[width=13.5cm]{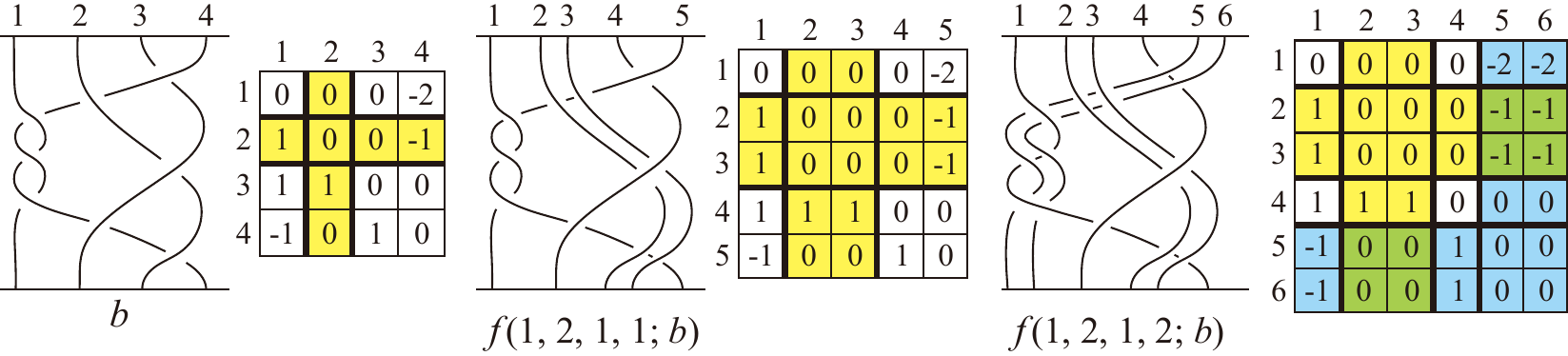}
\caption{$C(b)$, $C(f(1,2,1,1;b))$, and $C(f(1,2,1,2;b))$.}
\label{fig-cb}
\end{figure}
\end{example}
\medskip

\subsection{Crossing matrix of cabled braids}

In Section \ref{subsection-integer}, we observed crossing matrices of integer-cabled braids. 
In this subsection, we investigate crossing matrices for general cabling. \\

The crossing matrix of the cabled braid $f(b_1, b_2, \dots , b_n ; b)$ is obtained as follows. 
Let $b_i \in \mathcal{B}_{k_i}$. 
For $C(b)$, apply the transformation of Proposition \ref{prop-CM-2} repeatedly to obtain $C(f(k_1, k_2, \dots , k_n;b))$. 
Then replace the $k_i \times k_i$ zero submatrix indexed by rows and columns $k_1+k_2+\dots +k_{i-1}+1, \dots , k_1+k_2+\dots +k_{i-1}+k_i$ with $C(b_i)$ for each $i \in \{ 1, 2, \dots , n\}$.

\medskip 
\begin{example}
Let $b$ be the 4-braid in Figure \ref{fig-cb}. 
The crossing matrix of the cabled braid $C(f(b_1, b_2, b_3, b_4 ;b))$ with $b_1=b_3=b_4=id_1 \in \mathcal{B}_1$, $b_2=\sigma_2^{-1} \sigma_1 \in \mathcal{B}_3$ is shown in Figure \ref{fig-cb2}. 
\begin{figure}[ht]
\centering
\includegraphics[width=12cm]{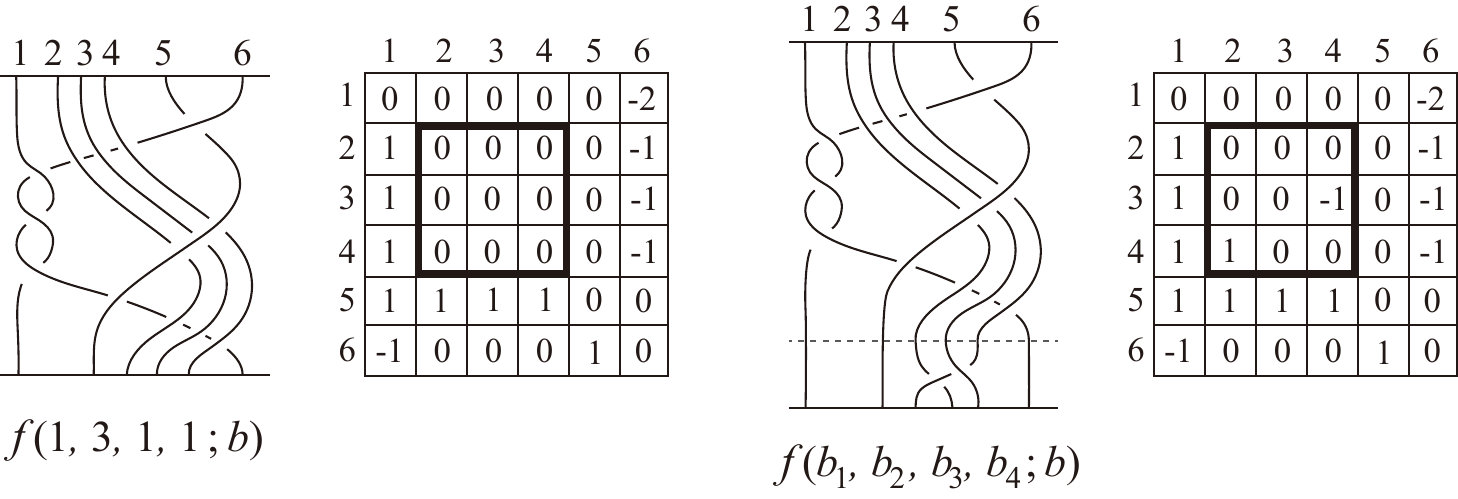}
\caption{$C(f(1,3,1,1;b))$ and $C(f(b_1, b_2, b_3, b_4 ;b))$ with $b_1=b_3=b_4=id_1 \in \mathcal{B}_1$, $b_2= \sigma_2^{-1} \sigma_1 \in \mathcal{B}_3$.}
\label{fig-cb2}
\end{figure}
\end{example}
\medskip

\begin{definition}
An $n \times n$ matrix $M$ is said to {\it admit a cable-block decomposition} if $M$ is in the form of $m \times m$ blocks ($1<m<n$) as 
\begin{align*}
M=\begin{bmatrix}
A_{11} & A_{12} & \dots & A_{1m} \\
A_{21} & A_{22} & \dots & A_{2m} \\
\vdots & \vdots & \ddots & \vdots \\
A_{m1} & A_{m2} & \dots & A_{mm} 
\end{bmatrix}
\end{align*}
such that the block $A_{ii}$ is a square matrix for each $i \in \{ 1, 2, \dots , m \}$ and all the entries in $A_{i j}$ are equal for each pair $i \neq j \in \{ 1, 2, \dots , m \}$. 
\end{definition}
\medskip 

\noindent The following proposition gives a necessary condition for a braid to be cabled. 

\medskip 
\begin{proposition}
If $b \in \mathcal{B}_n$ is cabled, then the crossing matrix $C(b)$ admits a cable-block decomposition. 
\label{prop-CM-cable}
\end{proposition}
\medskip 

\begin{proof}
Let $b \in \mathcal{B}_n$ be a cabled braid with $m$ tubes. 
Suppose that the $i^{th}$ tube has $k_i$ strands. 
Partition the rows and columns of $C(b)$ according to the $m$ tubes, whose sizes are $k_1, k_2, \dots , k_m$. 
Then each $A_{ii}$ block is a $k_i \times k_i$ square matrix. 
Every pair consisting of a strand in the $i^{th}$ tube and a strand in the $j^{th}$ tube has the same crossing information. 
Hence all the entries in $A_{i j}$ are equal for each pair $i \neq j$. 
Therefore, $C(b)$ admits an $m \times m$ cable-block decomposition. 
\end{proof}
\medskip

\begin{example}
The crossing matrix $C(b)$ of the braid $b$ in Figure \ref{fig-cb} does not admit a cable-block decomposition with any of the block-size patterns $(1,3)$, $(2,2)$, $(3,1)$, $(1,1,2)$, $(1,2,1)$, $(2,1,1)$. 
Hence, $b$ is not cabled by the contrapositive of Proposition \ref{prop-CM-cable}. 
\label{ex-non-cabled}
\end{example}
\medskip

\subsection{Proofs}

In this subsection, we prove Theorem \ref{thm-CM-red}, Corollaries \ref{cor-red2} and \ref{cor-anosov}. 
The following lemma will be used for proving Theorem \ref{thm-CM-red}. 

\medskip 
\begin{lemma}
An $n \times n$ zero-diagonal matrix $M$ is the crossing matrix of a standard form of a reducible pure $n$-braid if and only if $M$ is symmetric and admits a cable-block decomposition. 
\label{lem-CM-ch}
\end{lemma}
\medskip 

\begin{proof}
Suppose that $a$ is in a standard form of a reducible pure braid. 
Since $a$ is pure, $C(a)$ is symmetric by Lemma \ref{lem-CM-pure}. 
Since $a$ is cabled, $C(a)$ admits a cable-block decomposition by Proposition \ref{prop-CM-cable}. 
Hence $C(a)$ satisfies the condition. 

Next, suppose that $M$ is a symmetric $n \times n$ matrix that admits a cable-block decomposition with $m \times m$ blocks ($1<m<n$) as follows. 
\begin{align*}
M = \begin{bmatrix}
A_{11} & A_{12} & \dots & A_{1m} \\
A_{21} & A_{22} & \dots & A_{2m} \\
\vdots & \vdots & \ddots & \vdots \\
A_{m1} & A_{m2} & \dots & A_{mm} 
\end{bmatrix}.
\end{align*}
Let $a_{ij}$ be the value of the entries of the submatrix $A_{ij}$ for each $i \neq j$. 
Since $M$ is symmetric, we have $a_{ij}=a_{ji}$ for each $i \neq j$. 
Let $T= [t_{ij}]$ be an $m \times m$ zero-diagonal matrix with $t_{ij}=a_{ij}$ for each $i \neq j$. 
Since $T$ is symmetric, there exists a pure braid, say $b$, such that $C(b)=T$ by Lemma \ref{lem-CM-pure}. 
Since $M$ is symmetric, the submatrices $A_{ii}$ are also symmetric. 
By Lemma \ref{lem-CM-pure}, again, there exists a pure $k_i$-braid $b_i$ such that $C(b_i)=A_{ii}$ for each $i$. 
Then, the matrix $M$ is the crossing matrix of the cabled braid $c=f(b_1, b_2, \dots , b_m;b)$. 
Note that $c$ is in a standard form of a reducible pure braid. 
\end{proof}
\medskip

\begin{example}
Let $M$ be a symmetric $7 \times 7$ matrix that admits a cable-block decomposition with $3 \times 3$ blocks with $2, 2, 3$ rows and columns as follows: 
\begin{align*}
M=
\begin{bmatrix}
0 & 1 & 1 & 1 & -1 & -1 & -1 \\
1 & 0 & 1 & 1 & -1 & -1 & -1 \\
1 & 1 & 0 & 2 & 2 & 2 & 2 \\
1 & 1 & 2 & 0 & 2 & 2 & 2 \\
-1 & -1 & 2 & 2 & 0 & 2 & 1 \\
-1 & -1 & 2 & 2 & 2 & 0 & -1 \\
-1 & -1 & 2 & 2 & 1 & -1 & 0 \\
\end{bmatrix}.
\end{align*}
Let 
\begin{align*}
T=
\begin{bmatrix}
0 & 1 & -1 \\
1 & 0 & 2 \\
-1 & 2 & 0 
\end{bmatrix}, \ A_{11}=
\begin{bmatrix}
0 & 1 \\
1 & 0
\end{bmatrix}, \ A_{22}= 
\begin{bmatrix}
0 & 2 \\
2 & 0
\end{bmatrix}, \ A_{33}=
\begin{bmatrix}
0 & 2 & 1 \\
2 & 0 & -1 \\
1 & -1 & 0
\end{bmatrix}.
\end{align*}
The braids $b=\sigma_1 \sigma_2^{-2}\sigma_1\sigma_2^4 \in \mathcal{B}_3$, $b_1 = \sigma_1^2, \ b_2= \sigma_1^4 \in \mathcal{B}_2$, and $b_3= \sigma_1^3 \sigma_2^2 \sigma_1 \sigma_2^{-2} \in \mathcal{B}_3$ satisfy\footnote{The BW-ladder diagram introduced in \cite{AY-5} is useful to construct a pure braid for given matrix in some cases.} 
$C(b)=T$, $C(b_1)=A_{11}$, $C(b_2)=A_{22}$, and $C(b_3)=A_{33}$. 
Hence, $f(b_1, b_2, b_3;b)$ has $M$ as its crossing matrix. 
\end{example}
\medskip

\noindent The converse of Proposition \ref{prop-CM-cable} does not hold as shown in the following example. 

\medskip 
\begin{example}
For the weaving braid $b=(\sigma_1 \sigma_2^{-1})^3 \in \mathcal{P}_3$, we have $C(b)=O$. 
Although this $C(b)$ admits a $1 \times 2$ cable-block decomposition, this braid $b$ is not reducible. 
Indeed, this braid can be classified as pseudo-Anosov (\cite{TM}). 
Thus, the crossing matrix does not completely determine whether a given pure braid is reducible or not. 
\end{example}
\medskip

\noindent Next, we prove Theorem \ref{thm-CM-red}. 

\medskip 
\begin{proof}[Proof of Theorem \ref{thm-CM-red}. ]
\begin{itemize}
\item[($\Rightarrow$)] Suppose that $A=C(b)$ for a reducible pure $n$-braid $b$. 
By Proposition \ref{prop-GM}, $b$ is conjugate to a standard form $b' \in \mathcal{P}_n$. 
By Lemma \ref{lem-CM-p-perm}, $A$ is permutation equivalent to the crossing matrix $C(b')=A'$. 
By Lemma \ref{lem-CM-ch}, $A'$ is symmetric and has a cable-block decomposition. 
Let $\rho$ be a permutation on $\{ 1, 2, \dots , n \}$ such that $\rho (A)=A'$. 
Choose a cable-block decomposition of $A'$ with $m$ blocks as 
\begin{align*}
A' = \begin{bmatrix}
A'_{11} & A'_{12} & \dots & A'_{1m} \\
A'_{21} & A'_{22} & \dots & A'_{2m} \\
\vdots & \vdots & \ddots & \vdots \\
A'_{m1} & A'_{m2} & \dots & A'_{mm} 
\end{bmatrix}.
\end{align*}
For $s \in \{ 1, 2, \dots , m \}$, take $N_s$ as the set of the row-indices of the entries of $A$ that correspond to the entries in the diagonal block $A'_{ss}$ of $A'$ with respect to $\rho$. 
Then, $N=N_1 \cup N_2 \cup \dots \cup N_m$ and $N_s \cap N_t= \emptyset$ for distinct $s, t \in \{ 1, 2, \dots , m \}$. 
\item[($\Leftarrow$)] Suppose that $A=[a_{ij}]$ is a zero-diagonal symmetric matrix and there is a partition $N_1\amalg N_2\amalg \dots \amalg N_m=N$ such that for any distinct $s, t \in \{ 1, 2, \dots , m \}$ and any $(i,j)$, $(i', j') \in N_s \times N_t$, $a_{ij}=a_{i'j'}$. 
Define an order ``$\ll$'' on $N$ as follows. 
For $a, b \in N$, if one of the following holds, we denote $a \ll b$. 
\begin{itemize}
\item[(i)] $a, b \in N_s$ for some $s \in \{ 1, 2, \dots , m \}$ and $a<b$. 
\item[(ii)] $a \in N_s$ and $b \in N_t$ for some $s, t \in \{ 1, 2, \dots , m\}$ with $s<t$. 
\end{itemize}
Then, all the elements of $N$ is ordered as follows: $p_1 \ll p_2 \ll \dots \ll p_n$ with $\{ p_1, p_2, \dots , p_n \} =N$. 
Then, $N_s= \{ p_{\sum_{u=1}^{s-1} |N_u|+1}, p_{\sum_{u=1}^{s-1} |N_u|+2}, \dots , p_{\sum_{u=1}^{s} |N_u|} \}$ for each $s \in \{ 1, 2, \dots , m \}$, and we denote it by $N_s= \{ p^{(s)}_1, p^{(s)}_2, \dots , p^{(s)}_{|N_s|} \}$, where $p^{(s)}_1<p^{(s)}_2<\dots < p^{(s)}_{|N_s|}$. 
Let $\rho$ be a permutation on $N$ such that $\rho(p_i)=i$ for $i \in \{ 1, 2, \dots , n \}$. 
Then $\rho(N_s)= \{ \sum_{u=1}^{s-1}|N_u|+1, \sum_{u=1}^{s-1} |N_u|+2, \dots , \sum_{u=1}^s |N_u| \}$ for each $s \in \{ 1, 2, \dots , m \}$, and we have 
\begin{align*}
\rho(A)= 
\begin{bmatrix}
A'_{11} & \dots & A'_{1m} \\
\vdots & \ddots & \vdots \\
A'_{m1} & \dots & A'_{mm}
\end{bmatrix}, \text{ where } \ A'_{st}= 
\begin{bmatrix}
a_{p^{(s)}_1 p^{(t)}_1} & \dots & a_{p^{(s)}_1 p^{(t)}_{|N_t|}} \\
a_{p^{(s)}_2 p^{(t)}_1} & \dots & a_{p^{(s)}_2 p^{(t)}_{|N_t|}} \\
\vdots & \ddots & \vdots \\
a_{p^{(s)}_{|N_s|} p^{(t)}_1} & \dots & a_{p^{(s)}_{|N_s|} p^{(t)}_{|N_t|}} \\
\end{bmatrix} \text{ for each } s,t\in\{1,2,\dots ,m\}. 
\end{align*}
If $s \neq t$, then $p^{(s)}_i \neq p^{(t)}_j$ for any $p^{(s)}_i \in N_s$ and any $p^{(t)}_j \in N_t$ (for any $i \in \{ 1, 2, \dots , |N_s| \}$ and any $j \in \{ 1, 2, \dots , |N_t| \}$). 
Hence,  if $s \neq t$, then $a_{p^{(s)}_i p^{(t)}_j}=a_{p^{(s)}_{i'} p^{(t)}_{j'}}$ for any $(p^{(s)}_{i}, p^{(t)}_{j})$, $(p^{(s)}_{i'}, p^{(t)}_{j'}) \in N_s \times N_t$, and we have $A'_{st}=[a]$ for some $a \in \mathbb{Z}$. 
Since $A$ is symmetric, $a_{p^{(t)}_j p^{(s)}_i}=a_{p^{(s)}_i p^{(t)}_j}$, and $A'_{ts}=(A'_{st})^T$. 
If $s=t$, 
\begin{align*}
A'_{ss}= 
\begin{bmatrix}
a_{p^{(s)}_1 p^{(s)}_1} & a_{p^{(s)}_1 p^{(s)}_2} & \dots & a_{p^{(s)}_{p_1} p^{(s)}_{|N_s|}} \\
a_{p^{(s)}_2 p^{(s)}_1} & a_{p^{(s)}_2 p^{(s)}_2} & \dots & a_{p^{(s)}_{p_2} p^{(s)}_{|N_s|}} \\
\vdots & \vdots & \ddots & \vdots \\
a_{p^{(s)}_{|N_s|} p^{(s)}_1} & a_{p^{(s)}_{|N_s|} p^{(s)}_2} & \dots & a_{p^{(s)}_{p_{|N_s|}} p^{(s)}_{|N_s|}} \\
\end{bmatrix}.
\end{align*}
Since $A$ is a zero-diagonal symmetric matrix, we have $a_{p^{(s)}_i p^{(s)}_i}=0$ and $a_{p^{(s)}_i p^{(s)}_j}=a_{p^{(s)}_j p^{(s)}_i}$ for any $p^{(s)}_i,  p^{(s)}_j \in N_s$. 
Thus, $A'_{ss}$is a zero-diagonal symmetric matrix. 
Since $A'_{ts}=(A'_{st})^T$, $\rho(A)$ is symmetric and admits a cable-block decomposition, and therefore $\rho(A)$ is the crossing matrix of a standard form $b'$ of a reducible pure braid by Lemma \ref{lem-CM-ch}. Take a braid $a \in \mathcal{B}_n$ whose braid permutation is $\rho$, and let $b=a^{-1}b'a$. Then $b \in \mathcal{P}_n$, and $b$ is reducible because reducibility is invariant under conjugation. Moreover, $C(b)=\rho^{-1}(C(b'))=\rho^{-1}(\rho(A))=A$. 
This completes the proof. 
\end{itemize}

\end{proof}
\medskip 

\noindent We prove Corollary \ref{cor-red2}. 

\medskip 
\begin{proof}[Proof of Corollary \ref{cor-red2}.] 
By Lemmas \ref{lem-CM-p-perm} and \ref{lem-CM-ch}, $A$ is permutation equivalent to a matrix $M$ with a cable-block decomposition. 
Choose a diagonal block $A_{ii}$, with size two or more, of the block decomposition of $M$. 
We can take such a block because $m<n$. 
Take $S$ as the set of row-indices of the entries of $A$ that correspond to the entries in $A_{ii}$. 
\end{proof}

\noindent We prove Corollary \ref{cor-anosov}. 

\medskip 
\begin{proof}[Proof of Corollary \ref{cor-anosov}. ]
By Lemma \ref{lem-CM-pure}, there exists a pure $n$-braid $b$ such that $C(b)=A$.
By Corollary \ref{cor-red2}, $b$ is not reducible. Thus, $b$ is pseudo-Anosov. 
(Note that every periodic pure $n$-braid is reducible for $n\ge 3$.)
\end{proof}
\medskip

\begin{example}
Let $b= \sigma_1\sigma_2\sigma_3^{-4}\sigma_2^{3}\sigma_1\sigma_2^{-1}\sigma_3^{-2}\sigma_2^{-3}\sigma_3^2 \in \mathcal{P}_4$. 
Then the crossing matrix is 
\begin{align*}
C(b)=
\begin{bmatrix}
0 & 1 & 2 & -2 \\
1 & 0 & -2 & -1 \\
2 & -2 & 0 & 1 \\
-2 & -1 & 1 & 0 
\end{bmatrix}. 
\end{align*}
Since all entries in each row of $C(b)$ are different, $b$ is pseudo-Anosov by Corollary \ref{cor-anosov}.
\label{ex-anosov}
\end{example}
\medskip

\section{Periodic braids}
\label{section-periodic}

Throughout this section, $n$ is a fixed integer with $n\geq 2$. 
We give a necessary condition for a pure braid to be periodic in terms of the crossing matrix. 
Let $\delta = \sigma_1 \sigma_2 \dots \sigma_{n-1}$, $\gamma = \sigma_1 \delta =  \sigma_1^2 \sigma_2 \dots \sigma_{n-1} \in \mathcal{B}_n$ as shown in Figure \ref{fig-gamma}. 
\begin{figure}[ht]
\centering
\includegraphics[width=6cm]{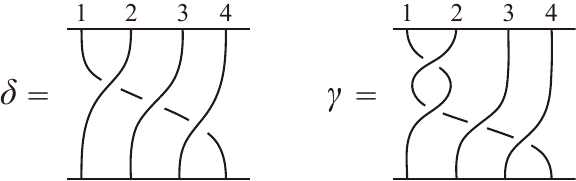}
\caption{$\delta$ and $\gamma \in \mathcal{B}_4$.}
\label{fig-gamma}
\end{figure}
An $n$-braid $b \in \mathcal{B}_n$ is {\it periodic} if there exists a non-zero integer $k \in \mathbb{Z}$ such that $b \stackrel{{ \mathrm{conj}}}{\backsim} \delta^k$ or $\gamma^k$. 
We have the following formula for the crossing matrices of $\delta^k$ and $\gamma^k$ for non-negative $k$.  

\medskip 
\begin{proposition}
Let $p, q$ be non-negative integers with $q \leq n-1$. 
Let $C(\delta^{pn+q})=[ m_{ij}]$, $C(\gamma^{p(n-1)+q})=[ n_{ij}]$. 
Then 
\begin{align*}
m_{ij} =
\begin{cases}
    0 & \text{if $i=j$,} \\
    p+1 & \text{if $j \leq q, \ i \neq j$, } \\
    p  & \text{otherwise,}
\end{cases}
\end{align*}
\begin{align*}
\ n_{ij} =
\begin{cases}
    0 & \text{if $i=j$,} \\
    p+1 & \text{if $(j=1, \ 2 \leq i \leq q+1)$ or $(2 \leq j \leq q+1, \ i \neq j)$,} \\
    p  & \text{otherwise.}
\end{cases}
\end{align*}
\label{prop-delta-gamma}
\end{proposition}
\medskip 

\begin{proof}
For $C(\delta)=[m_{i j}]$ and $C(\gamma)=[n_{i j}]$, we have 
\begin{align*}
m_{ij} =
\begin{cases}
    1 & \text{if $j=1, \ i \neq j \ $, } \\
    0 & \text{otherwise,}
\end{cases} 
\ n_{ij} =
\begin{cases}
    1 & \text{if $(j=1, \ i=2)$ or $(j=2, \ i \neq j)$,} \\
    0 & \text{otherwise.}
\end{cases}
\end{align*}
By Proposition \ref{prop-C-br}, we obtain the formula for $\delta^k$ and $\gamma^k$. 
\end{proof}
\medskip 

\begin{example}
When $\delta, \ \gamma \in \mathcal{B}_4$, we have 
\begin{align*}
C(\delta)=
\begin{bmatrix}
0 & 0 & 0 & 0 \\
1 & 0 & 0 & 0 \\
1 & 0 & 0 & 0 \\
1 & 0 & 0 & 0 
\end{bmatrix}, 
C(\delta^2)=
\begin{bmatrix}
0 & 1 & 0 & 0 \\
1 & 0 & 0 & 0 \\
1 & 1 & 0 & 0 \\
1 & 1 & 0 & 0 
\end{bmatrix}, 
C(\delta^3)=
\begin{bmatrix}
0 & 1 & 1 & 0 \\
1 & 0 & 1 & 0 \\
1 & 1 & 0 & 0 \\
1 & 1 & 1 & 0 
\end{bmatrix}, 
C(\delta^4)=
\begin{bmatrix}
0 & 1 & 1 & 1 \\
1 & 0 & 1 & 1 \\
1 & 1 & 0 & 1 \\
1 & 1 & 1 & 0 
\end{bmatrix}, 
C(\delta^{10})=
\begin{bmatrix}
0 & 3 & 2 & 2 \\
3 & 0 & 2 & 2 \\
3 & 3 & 0 & 2 \\
3 & 3 & 2 & 0 
\end{bmatrix}, 
\end{align*}
\begin{align*}
C(\gamma)=
\begin{bmatrix}
0 & 1 & 0 & 0 \\
1 & 0 & 0 & 0 \\
0 & 1 & 0 & 0 \\
0 & 1 & 0 & 0 
\end{bmatrix}, 
\ C(\gamma^2)=
\begin{bmatrix}
0 & 1 & 1 & 0 \\
1 & 0 & 1 & 0 \\
1 & 1 & 0 & 0 \\
0 & 1 & 1 & 0 
\end{bmatrix}, 
\ C(\gamma^3)=
\begin{bmatrix}
0 & 1 & 1 & 1 \\
1 & 0 & 1 & 1 \\
1 & 1 & 0 & 1 \\
1 & 1 & 1 & 0 
\end{bmatrix}, 
\ C(\gamma^8)=
\begin{bmatrix}
0 & 3 & 3 & 2 \\
3 & 0 & 3 & 2 \\
3 & 3 & 0 & 2 \\
2 & 3 & 3 & 0 
\end{bmatrix}. 
\end{align*}
\end{example}
\medskip

\noindent Throughout this section, $D$ denotes an $n \times n$ matrix such that the diagonal entries are zero and the other entries are 1. 
For the purified crossing matrix $C^{\text{pur}}(b)$, we have the following. 

\medskip 
\begin{proposition}
The purified crossing matrices of $\delta^k$ and $\gamma^k \in \mathcal{B}_n$ are given as $C^{\text{pur}}(\delta^k)=pD$, $C^{\text{pur}}(\gamma^k)=qD$ for some $p, q \in \mathbb{Z}$. 
In particular, $C^{\text{pur}}(\delta^k)=kD$ if $\text{gcd}(k,n)=1$, and $C^{\text{pur}}(\gamma^k)=kD$ if $\text{gcd}(k, n-1)=1$. 
\label{prop-c-4}
\end{proposition}
\medskip 

\begin{proof}The orders of the braid permutations of $\delta$ and $\gamma$ are
$n$ and $n-1$, respectively. Hence $\delta^n$ and
$\gamma^{n-1}$ are pure braids. First, consider $\delta^k$. Let $m=\gcd(k,n)$. Then the order of the
braid permutation of $\delta^k$ is $n/m$. Hence 
$C^{\mathrm{pur}}(\delta^k)=C\left((\delta^k)^{n/m}\right)=C(\delta^{kn/m})$. 
If $k\geq 0$, then Proposition \ref{prop-delta-gamma} implies that $C(\delta^{kn/m})=\frac{k}{m}D$. 
If $k<0$, let $k'=-k>0$. Then, $\delta^{kn/m}=(\delta^{k'n/m})^{-1}$. By the case $k'>0$, we have $C(\delta^{k'n/m})=\frac{k'}{m}D$. Since $\delta^{k'n/m}$ is pure, Proposition \ref{prop-CM-inverse} implies that taking the inverse changes the sign of the crossing matrix. Therefore, $C(\delta^{kn/m})=-\frac{k'}{m}D=\frac{k}{m}D$. Thus, $C^{\mathrm{pur}}(\delta^k)=\frac{k}{m}D$. We can prove this for $\gamma^k$ in the same way. If $m=\gcd(k,n-1)$, then the
order of the braid permutation of $\gamma^k$ is $(n-1)/m$, and hence $C^{\mathrm{pur}}(\gamma^k)=C\left((\gamma^k)^{(n-1)/m}\right)=C(\gamma^{k(n-1)/m})=\frac{k}{m}D$. 
Therefore the purified crossing matrices of $\delta^k$ and $\gamma^k$
are of the form $pD$ and $qD$, respectively, for some $p,q\in\mathbb Z$.
In particular, if $\gcd(k,n)=1$, then $C^{\mathrm{pur}}(\delta^k)=kD$, and if $\gcd(k,n-1)=1$, then $C^{\mathrm{pur}}(\gamma^k)=kD$. This completes the proof.
\end{proof}
\medskip

\medskip 
\begin{example}
Let $p \in \mathbb{Z}$. 
The purified crossing matrices of $\delta^k$ and $\gamma^k \in \mathcal{B}_4$ are as follows. 
\begin{align*}
\begin{cases}
C^{\text{pur}}(\delta^{4p})=C(\delta^{4p})=pD, \\
C^{\text{pur}}(\delta^{4p+1})=C((\delta^{4p+1})^4)=(4p+1)D, \\
C^{\text{pur}}(\delta^{4p+2})=C((\delta^{4p+2})^2)=(2p+1)D, \\
C^{\text{pur}}(\delta^{4p+3})=C((\delta^{4p+3})^4)=(4p+3)D, \\
\end{cases}
\begin{cases}
C^{\text{pur}}(\gamma^{3p})=C(\gamma^{3p})=pD, \\
C^{\text{pur}}(\gamma^{3p+1})=C((\gamma^{3p+1})^3)=(3p+1)D, \\
C^{\text{pur}}(\gamma^{3p+2})=C((\gamma^{3p+2})^3)=(3p+2)D.
\end{cases}
\end{align*}
\label{ex-c-4}
\end{example}
\medskip

\begin{example}
Let $b \in \mathcal{B}_4$ be the 4-braid shown in Figure \ref{fig-cb} with order of braid permutation 2. 
Then 
\begin{align*}
C(b)=
\begin{bmatrix}
0 & 0 & 0 & -2 \\
1 & 0 & 0 & -1 \\
1 & 1 & 0 & 0 \\
-1 & 0 & 1 & 0 
\end{bmatrix}, 
\ C^{\text{pur}}(b)=C(b^2)=
\begin{bmatrix}
0 & 1 & 0 & -3 \\
1 & 0 & 1 & 0 \\
0 & 1 & 0 & 1 \\
-3 & 0 & 1 & 0 
\end{bmatrix}.  
\end{align*}
Among all powers of $\delta$ and $\gamma$, only braids of the form $\delta^{4p+2}$ with $p \in \mathbb{Z}$ have braid permutation of order 2. 
By Proposition \ref{prop-c-4}, $C^{\text{pur}}(\delta^{4p+2})=(2p+1)D$ and this is not permutation equivalent to $C^{\text{pur}}(b)$. 
Hence, $b$ is not periodic because $b$ is not conjugate to $\delta^k$ or $\gamma^k$ for any $k$. 
\end{example}
\medskip 

\noindent We have the following lemma. 

\medskip 
\begin{lemma}
If a pure braid $b\in {\mathcal P}_n$ is periodic, then $C(b)=pD$ for some $p\in {\mathbb Z}$. 
\label{lem-periodic}
\end{lemma}
\medskip 

\begin{proof}
Since $b$ is periodic, $b \stackrel{{ \mathrm{conj}}}{\backsim} \delta^k$ or $\gamma^k$ for some $k\in {\mathbb Z}$. 
Since $b$ is pure, if $b \stackrel{{ \mathrm{conj}}}{\backsim} \delta^k$, then $k \equiv 0 \pmod{n}$, and if $b \stackrel{{ \mathrm{conj}}}{\backsim} \gamma^k$, then $k \equiv 0 \pmod{n-1}$. 
By Proposition \ref{prop-c-4}, then, $C(\delta^k)=C^{\text{pur}}(\delta^k)=pD$ for some $p$, or $C(\gamma^k)=C^{\text{pur}}(\gamma^k)=pD$ for some $p$. 
Any permutation-equivalent matrix to $pD$ is the matrix $pD$ itself, and therefore we have $C(b)=pD$. 
\end{proof}
\medskip


\appendix
\section{Enhanced conjugacy invariants}
\label{section-conjugacy}

In this appendix, we use cabling to enhance conjugacy invariants.

\subsection{Properties of cabled braids}

In this subsection, we explore properties of cabled braids that are used to prove Proposition \ref{prop-p-Fb}. 
For the braid product, we have the following lemma. 

\medskip 
\begin{lemma}
Let $b, b' \in \mathcal{B}_n$. 
Let $\varphi$ be the braid permutation of $b$. 
Let $b_i, b'_{\varphi(i)} \in \mathcal{B}_{k_i}$ ($i=1, 2, \dots , n$). 
We have 
\begin{align*}
f(b_1, b_2, \dots , b_n;b) f(b'_{1}, b'_{2}, \dots , b'_{n};b') =
f(b_1 b'_{\varphi(1)}, b_2 b'_{\varphi(2)}, \dots , b_n b'_{\varphi(n)};bb').
\end{align*}
\label{lem-product}
\end{lemma}

\begin{proof}
By the braid product on the cabled braids, the $i^{th}$ tube of $f(b_1, b_2, \dots , b_n;b)$ is connected to the $\varphi(i)^{th}$ tube of the other cabled braid. 
They have the same number of strands. 
Within each tube, the braid product $b_i b'_{\varphi(i)}$ can be performed independently. 
In total, we have the cabling of $bb'$. 
\end{proof}
\medskip 

\begin{example}
Let $b=\sigma_1$, $b'=\sigma_2^{-1} \in \mathcal{B}_3$. 
Then the braid permutation of $b$ is $\varphi (1,2,3)=(2,1,3)$. 
Let $b_1=b'_2=id_1 \in \mathcal{B}_1$, $b_2=\sigma_1$, $b'_1=\sigma_1^2 \in \mathcal{B}_2$, $b_3=\sigma_2 \sigma_1$, $b'_3=\sigma_2^{-1} \in \mathcal{B}_3$. 
We have $f(id_1, \sigma_1, \sigma_2\sigma_1;b)f(\sigma_1^2, id_1, \sigma_2^{-1};b')=f(id_1, \sigma_1^3, \sigma_2\sigma_1\sigma_2^{-1};bb')$. 
(See Figure \ref{fig-product}.) 
\begin{figure}[ht]
\centering
\includegraphics[width=8cm]{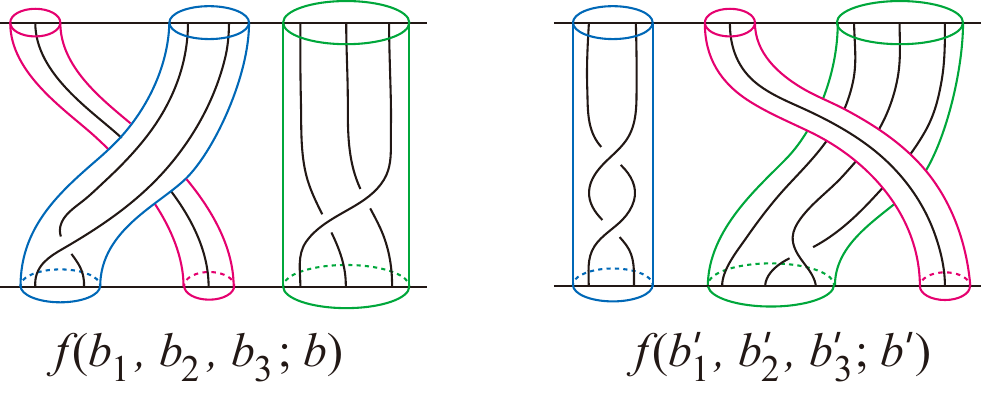}
\caption{Cabled braids. }
\label{fig-product}
\end{figure}
\end{example}
\medskip

\noindent For the inverse, we have the following lemma.

\medskip 
\begin{lemma}
Let $a \in \mathcal{B}_n$, $b_1, b_2, \dots , b_n \in \mathcal{B}$. 
We have $f(b_1, b_2, \dots , b_n;a)^{-1}=f(b_{\varphi^{-1}(1)}^{-1}, b_{\varphi^{-1}(2)}^{-1}, \dots , b_{\varphi^{-1}(n)}^{-1}; a^{-1})$, where $\varphi$ is the braid permutation of $a$. 
\label{lem-inverse}
\end{lemma}
\medskip 

\begin{proof}
Let $b_i \in \mathcal{B}_{k_i}$ ($i=1, 2, \dots , n$). 
We observe that $f(b^{-1}_{\varphi^{-1}(1)}, b^{-1}_{\varphi^{-1}(2)}, \dots , b^{-1}_{\varphi^{-1}(n)}; a^{-1}) f(b_1, b_2, \dots , b_n;a)$ is transformed tubewise into the trivial braid $id \in \mathcal{B}_{k_1+k_2+\dots +k_n}$ by the transformation corresponding to that of $a^{-1}a$ to $id \in \mathcal{B}_n$ and $b^{-1}_ib_i=id_{k_i}$ for each $i$. 
\end{proof}
\medskip 

\begin{example}
For the braid $a \in \mathcal{B}_3$ shown in Figure \ref{fig-aa} with braid permutation $\varphi (1,2,3)=(3,1,2)$, we have $(f(id_1,id_2,id_3;a))^{-1}=f(id_2,id_3,id_1; a^{-1})$. 
\begin{figure}[ht]
\centering
\includegraphics[width=9cm]{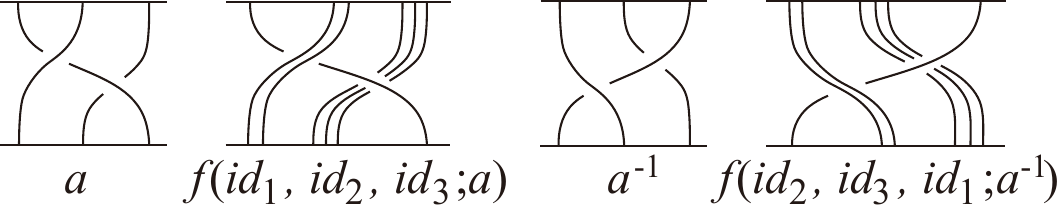}
\caption{$(f(id_1,id_2,id_3;a))^{-1}=f(id_2,id_3,id_1; a^{-1})$.}
\label{fig-aa}
\end{figure}
\end{example}
\medskip

\noindent 
We have the following proposition. 

\medskip 
\begin{proposition}
For each pair of conjugate pure braids $b \stackrel{{ \mathrm{conj}}}{\backsim} b' \in \mathcal{P}_n$, there exists a permutation $\rho$ on $\{ 1, 2, \dots , n \}$ such that $f(b_1, b_2, \dots , b_n ; b) \stackrel{{ \mathrm{conj}}}{\backsim} f(b_{\rho (1)}, b_{\rho (2)}, \dots , b_{\rho (n)}; b')$ for any $b_1, b_2, \dots , b_n \in \mathcal{B}$. 
Moreover, if $a \in \mathcal{B}_n$ conjugates $b$ to $b'$, then $f(b_1, b_2, \dots , b_n; a)$ conjugates  $f(b_1, b_2, \dots , b_n ; b)$ to $ f(b_{\rho (1)}, b_{\rho (2)}, \dots , b_{\rho (n)}; b')$. 
\label{prop-conj-perm}
\end{proposition}
\medskip

\begin{proof}
Suppose that $b'=a^{-1}ba$ for $a \in \mathcal{B}_n$. 
Let $\varphi$ be the braid permutation of $a$. 
In the form of $a^{-1}ba$, the $i^{th}$ strand of $b$ is the $\varphi (i)^{th}$ strand of $a^{-1}ba$. 
Equivalently, the $j^{th}$ strand of $b'=a^{-1}ba$ is the $\varphi^{-1}(j)^{th}$ strand of $b$. 
Set $\varphi^{-1}=\rho$. 
Then $f(b_1, b_2, \dots , b_n; a)$ conjugates $f(b_1, b_2, \dots , b_n; b)$ to 
\begin{align*}
& \{ f(b_1, b_2, \dots , b_n; a) \}^{-1} f(b_1, b_2, \dots , b_n; b) f(b_1, b_2, \dots , b_n; a) \\
= & f( b^{-1}_{\varphi^{-1}(1)}, b^{-1}_{\varphi^{-1}(2)}, \dots , b^{-1}_{\varphi^{-1}(n)} ;a^{-1}) f(b_1, b_2, \dots , b_n; b) f(b_1, b_2, \dots , b_n; a) \\
= & f( b^{-1}_{\varphi^{-1}(1)}b^{2}_{\varphi^{-1}(1)}, b^{-1}_{\varphi^{-1}(2)}b^{2}_{\varphi^{-1}(2)}, \dots , b^{-1}_{\varphi^{-1}(n)}b^{2}_{\varphi^{-1}(n)}; a^{-1}ba) \\
= & f( b_{\rho (1)},  b_{\rho (2)}, \dots ,  b_{\rho (n)};b')
\end{align*}
By Lemmas \ref{lem-product} and \ref{lem-inverse}. 
\end{proof}
\medskip 

\begin{example}
Let $b,  b'=a^{-1}ba \in \mathcal{P}_3$ be the braids illustrated in Figure \ref{fig-bb}. 
We have $f(b_1, b_2, b_3 ;b) \stackrel{{ \mathrm{conj}}}{\backsim} f(b_1, b_3, b_2 ;b')$ for any $b_1, b_2, b_3 \in \mathcal{B}$. 
\begin{figure}[ht]
\centering
\includegraphics[width=12cm]{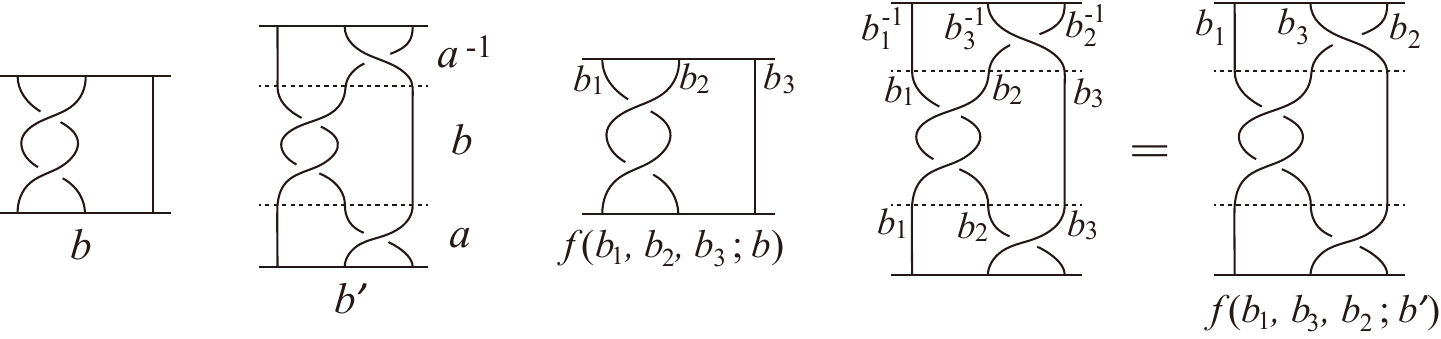}
\caption{$f(b_1, b_2, b_3 ;b) \stackrel{{ \mathrm{conj}}}{\backsim} f(b_1, b_3, b_2 ;b')$.}
\label{fig-bb}
\end{figure}
\label{ex-conj}
\end{example}
\medskip

\subsection{Enhanced conjugacy invariants}

Let $\pi (1, 2, \dots , n)=(\pi (1), \pi(2), \dots , \pi(n))$ be a permutation on $\{ 1, 2, \dots , n \}$. 
Let $\pi(\mathbf{\beta})=(b_{\pi(1)}, b_{\pi(2)}, \dots , b_{\pi(n)})$ denote the tuple of braids that is obtained from $\mathbf{\beta}=(b_1, b_2, \dots , b_n)$ by applying the permutation $\pi$. 
Let $g$ be a conjugacy invariant of braids. 
We attempt to enhance $g$ by cabling. 
We define a multiset consisting of $n!$ entries as follows. 
$$F_{\mathbf{\beta}}(g;b)= \{ g(f(\pi(\mathbf{\beta});b)) \ | \ \pi \text{ is a permutation on } \{ 1, 2, \dots , n \} \} .$$

\noindent Now we prove Proposition \ref{prop-p-Fb}. 

\medskip 
\begin{proof}[Proof of Proposition \ref{prop-p-Fb}.]
Suppose that $a$ conjugates $b$ to $b' \in \mathcal{P}_n$. 
Let $\rho$ be the permutation that is mentioned in Proposition \ref{prop-conj-perm}. 
Then each braid $f(\pi(\mathbf{\beta});b)$ that is obtained from $b$ by the cabling of $\pi(\mathbf{\beta})$ is conjugate to the braid $f(\rho (\pi(\mathbf{\beta}));b')$, and they have the same value of the invariant $g$ by conjugacy. 
Since $\rho$ gives a one-to-one correspondence, we have $F_{\mathbf{\beta}}(g;b)=F_{\mathbf{\beta}}(g;b')$. 
\end{proof}
\medskip

\noindent Let $r$ be the order of braid permutation of $b \in \mathcal{B}_n$. 
Then $b^r \in \mathcal{P}_n$. 
We define $F^{\text{pur}}_{\mathbf{\beta}}(g;b)$ as $F_{\mathbf{\beta}}(g;b^r)$. 
We have the following corollary. 

\medskip 
\begin{corollary}
Let $b$, $b' \in \mathcal{B}_n$. 
If $b$ and $b'$ are conjugate, then $F^{\text{pur}}_{\mathbf{\beta}}(g;b)=F^{\text{pur}}_{\mathbf{\beta}}(g;b')$ for any $n$-tuple of braids $\mathbf{\beta}$ and conjugacy invariant $g$. 
\label{cor-b-Fb}
\end{corollary}
\medskip

\begin{proof}
Let $r$ and $r'$ be the orders of braid permutation of $b$ and $b'$, respectively. 
It follows from Proposition \ref{prop-p-Fb} by $r =r'$ and $b^{r} \stackrel{{ \mathrm{conj}}}{\backsim} (b')^{r'} \in \mathcal{P}_n$.
\end{proof}
\medskip

\subsection{The $(2,l)$ sets}

For a braid diagram $B=\sigma_{i_1}^{\varepsilon_1}\sigma_{i_2}^{\varepsilon_2} \dots \sigma_{i_m}^{\varepsilon_m}$, the {\it exponent sum} or the {\it writhe of $B$} is the sum of the crossing signs $w(B)= \sum_{k=1}^{m}\varepsilon_k$. 
For example, the braid diagram $B$ in Figure \ref{fig-BD} has $w(B)=1$. 
The exponent sum $w$ does not depend on the choice of a diagram $B$ of a braid $b$, namely, $w(b)=w(B)$ is a braid invariant. 
Moreover, it is well known that $w$ is a conjugacy invariant, namely, $w(b)=w(a^{-1}ba)$ holds for any $a, b \in \mathcal{B}_n$. \\

The operation from $b \in \mathcal{B}_n$ to $f(b_1, b_2, \dots , b_n;b)$ with $b_1= \dots = b_{k-1}=b_{k+1}= \dots = b_n = id_1$ and $b_k=id_2$ is called a {\it 2-cabling for the $k^{th}$ strand of $b$}. 
For $b \in \mathcal{B}_n$ and $l \in \{ 0, 1, 2, \dots , n \}$, let $f_{i_1 i_2 \dots i_l}(2;b)$ denote the braid that is obtained from $b$ by 2-cabling the ${i_1}^{th}, {i_2}^{th}, \dots$, and ${i_l}^{th}$ strands of $b$.

\medskip 
\begin{definition}
Let $b \in \mathcal{B}_n$. 
Let $w(b)$ be the exponent sum of $b$. 
The {\it $(2,l)$ set of $b$}, denoted by $W_l(b)$, is the multiset of the values of $w$ for all the $\binom{n}{l}$ braids obtained from $b$ by 2-cabling $l$ strands of $b$. 
Namely, 
$$W_l(b)= \{ w(f_{i_1 i_2 \dots i_l}(2;b)) \ | \ 1 \leq i_1 < i_2 < \dots < i_l \leq n \} .$$
\end{definition}
\medskip

\begin{remark}
When $l=0$, we have $W_0(b)= \{ w(b) \}$, that is, the $(2,0)$ set $W_0(b)$ has the same information as the exponent sum $w(b)$. 
When $l=1$, namely, $W_1(b)=\{ w(f_1(2;b)), w(f_2(2;b)), \dots , w(f_n(2;b)) \}$, we can calculate the value of each $w(f_i(2;b))$ as $w(f_i(2;b))=w(b)+w(s_i)$, where $w(s_i)$ is the sum of the signs of the crossings on the $i^{th}$ strand $s_i$ in a diagram of $b$. 
\label{rem-l01}
\end{remark}
\medskip

\noindent For pure braids, we have the following corollary. 

\medskip 
\begin{corollary}
Let $b$, $b' \in \mathcal{P}_n$. 
If $b$ and $b'$ are conjugate, then $W_l(b)=W_l(b')$ for any $l \in \{ 0, 1, 2, \dots , n \}$. 
\label{cor-p-2l}
\end{corollary}
\medskip

\begin{proof}
This holds in the same way as Proposition \ref{prop-p-Fb}. 
\end{proof}
\medskip

\noindent The $(2,1)$ set distinguishes some non-conjugate braids as shown in the following example. 

\medskip 
\begin{example}
Let $b, b' \in \mathcal{P}_3$ be as depicted in Figure \ref{fig-3-2}. 
Then $W_1(b)= \{ w(f_1(2; b)), w(f_2(2; b)), w(f_3(2; b)) \} =\{ 2,2,-4 \}$. 
For $b'$, we have $w(f_1(2; b'))=0 \not\in W_1(b)$ and therefore $W_1(b') \neq W_1(b)$. 
This implies that $b' \not\stackrel{{ \mathrm{conj}}}{\backsim} b$ by the contrapositive of Corollary \ref{cor-p-2l}. 
Note that $b$ and $b'$ have the same exponent sum as $w(b')=w(b)=0$. 
\begin{figure}[ht]
\centering
\includegraphics[width=10cm]{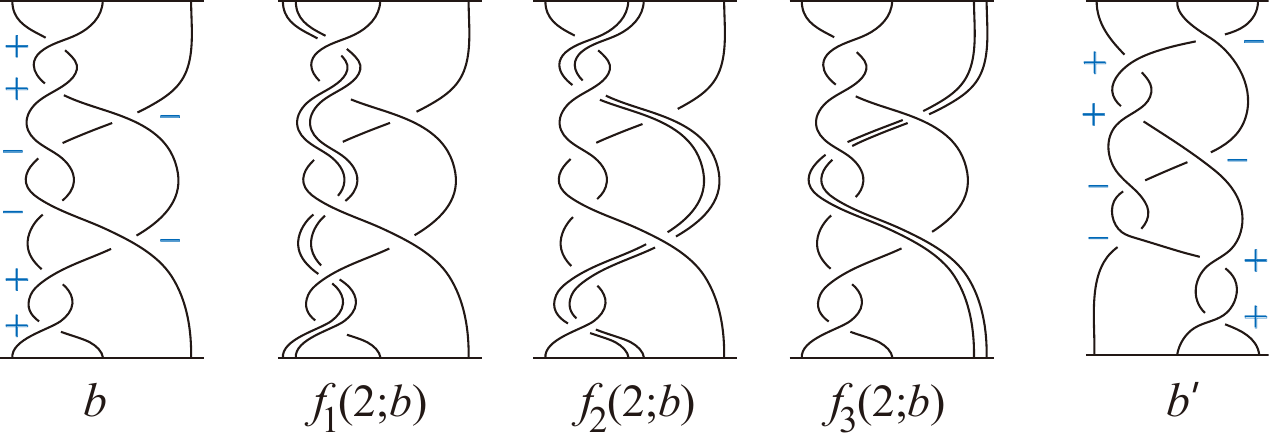}
\caption{A pair of pure 3-braids $b, b'$ with $w(b)=w(b')=0$ and $W_1(b)= \{ 2, 2, -4 \} \neq W_1(b')=\{ 0, -2, 2 \}$.  }
\label{fig-3-2}
\end{figure}
\label{ex-w1}
\end{example}
\medskip 

\noindent The $(2,2)$ set $W_2$ distinguishes some non-conjugate pure braids with the same $w$ and $W_1$ as follows. 

\medskip 
\begin{example}
The 4-braids $b, b' \in \mathcal{P}_4$ shown in Figure \ref{fig-4} have $w(b)=w(b')=0$, $W_1(b)=W_1(b')=\{ -4, 0, 2, 2 \}$. 
As for $W_2$, $b$ has $W_2(b) = \{ 2, 6, -2, 4, -6, -4 \}$. 
On the other hand, $b'$ has $w(f_{1 3}(2;b'))=-8$ and therefore $W_2(b) \neq W_2(b')$. 
Hence, $b \not\stackrel{{ \mathrm{conj}}}{\backsim} b'$. 
\begin{figure}[ht]
\centering
\includegraphics[width=12cm]{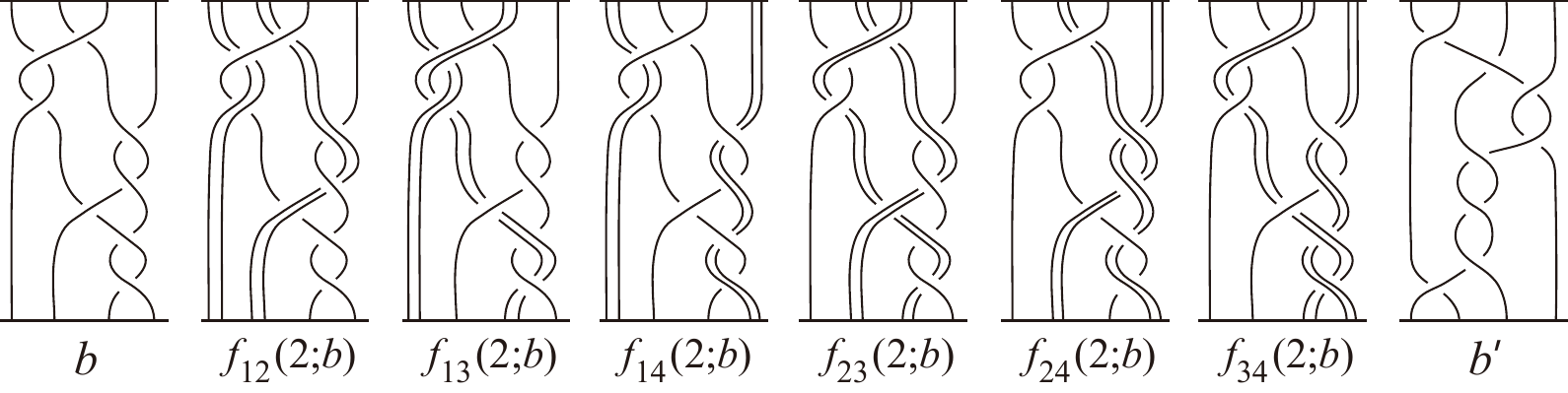}
\caption{A pair of pure 4-braids $b, b'$ with $w(b)=w(b')$, $W_1(b)=W_1(b')$, and $W_2(b)= \{ 2, 6, -2, 4, -6, -4 \} \neq W_2(b')=\{ 4, -8, 4, -2, 4, -2 \}$.  }
\label{fig-4}
\end{figure}
\label{ex-w2}
\end{example}
\medskip

\noindent Examples \ref{ex-w1} and \ref{ex-w2} show the effectiveness of cabling. 
The following proposition shows that $W_l$ determines $w$, namely, there are no pairs of braids $b, b' \in \mathcal{B}_n$ such that $w(b) \neq w(b')$ and $W_l (b)=W_l(b')$. 

\medskip 
\begin{proposition}
If $W_l(b)=W_l(b')$, then $w(b)=w(b')$ for any $l \in \{ 0, 1, 2, \dots , n \}$. 
\label{prop-W-l}
\end{proposition}
\medskip 

\begin{proof}
\begin{itemize}
\item Suppose that $l=0$. 
Then $W_0(b)= \{  w(b) \}$. 
Hence, if $W_0(b)=W_0(b')$, then $w(b)=w(b')$. 
\item Suppose that $l=1$. 
From Remark \ref{rem-l01}, we have the sum of all the entries $w(f_i(2;b))$ in $W_1(b)$ as $\sum_{i=1}^n w(f_i(2;b))=\sum_{i=1}^n(w(b)+w(s_i))=n \times w(b)+\sum_{i=1}^n w(s_i)=(n+2)w(b)$, that is, the sum is a non-zero constant multiplication of $w(b)$.  
If $W_1(b)=W_1(b')$ for $b, b' \in \mathcal{B}_n$, then the sums of entries are equal, namely, $(n+2)w(b)=(n+2)w(b')$ and therefore $w(b)=w(b')$. 
\item Suppose that $2 \leq l \leq n-2$. 
Let $p$ be a crossing between $s_i$ and $s_j$ ($i \neq j$) in a diagram of $b$. 
In $W_l(b)$, there are $\binom{n-2}{l}$ cases that none of $s_i$ or $s_j$ is 2-cabled. 
There are $\binom{n-2}{l-1}$ cases that exactly one of $s_i$ and $s_j$ is 2-cabled. 
There are $\binom{n-2}{l-2}$ cases that both of $s_i$ and $s_j$ are 2-cabled. 
Let $\varepsilon (p)$ be the sign of $p$. 
The total sum of all the entries in $W_l(b)$ is $\sum_p \varepsilon (p) \left\{ \binom{n-2}{l}+2 \binom{n-2}{l-1} \times 2 + 4 \binom{n-2}{l-2} \right\}=w(b) \left\{ \binom{n-2}{l}+4 \binom{n-2}{l-1} + 4 \binom{n-2}{l-2} \right\}$, that is, the total sum is a non-zero constant multiplication of $w(b)$. 
Hence, if $W_l(b)=W_l(b')$, then $w(b)=w(b')$. 
\item Suppose that $l=n-1$. 
The sum of all the entries in $W_{n-1}(b)$ is $\sum_{i=1}^n (4w(b)-2w(s_i))=(4n-4)w(b)$. 
\item Suppose that $l=n$. 
Then $W_n(b)= \{ 4 w(b) \}$. 
\end{itemize}
\end{proof}
\medskip

\noindent The $(2,l)$ set is not effective for some braids as follows. 

\medskip 
\begin{example}
Let $b=(\sigma_1 \sigma_2^{-1} \sigma_3 \dots \sigma_{n-1}^{-1})^n \in \mathcal{P}_n$ be a weaving braid with an odd number $n$. 
We have $w(f_{i_1i_2 \dots i_l}(2;b))=0$ for every $l \in \{ 0, 1, 2, \dots , n \}$ and every choice of $1 \leq i_1 < i_2 < \dots < i_l \leq n$ since each pair of strands has two crossings with opposite signs (See Proposition 2.1 in \cite{ASA}). 
Hence, $W_l(b)$ is a multiset of $\binom{n}{l}$ zeros, which is the same as $W_l(id_n)$. 
Since $b \neq id_n$, and the identity braid is conjugate to only itself, $b$ and $id_n$ are not conjugate. 
In this case, the $(2,l)$ sets cannot distinguish $b$ from $id$ up to conjugacy. 
\label{ex-weaving}
\end{example}
\medskip

\begin{definition}
For a braid $b \in \mathcal{B}_n$, let $r$ be the order of the braid permutation of $b$. 
The {\it purified $(2,l)$ set of $b$}, denoted by $W^{\text{pur}}_l(b)$, is defined as $W^{\text{pur}}_l(b)=W_l(b^r)$. 
\end{definition}
\medskip 

\noindent We have the following corollary. 

\medskip 
\begin{corollary}
Let $b$, $b' \in \mathcal{B}_n$. 
If $b$ and $b'$ are conjugate, then $W^{\text{pur}}_l(b)=W^{\text{pur}}_l(b')$ for each $l \in \{ 0, 1, 2, \dots , n \}$. 
\label{cor-b-2l}
\end{corollary}
\medskip

\begin{proof}
This holds in the same way as Corollary \ref{cor-b-Fb}. 
\end{proof}
\medskip 

\noindent For the purified version, we have the following example that shows the effectiveness of 2-cabling. 

\medskip
\begin{example}
Let $b=\sigma_1 \sigma_2 \sigma_3^{-1}$, $b'= \sigma_1 \sigma_2^{-1} \sigma_3 \in \mathcal{B}_4$. 
Note that $b$ and $b'$ have the same exponent sum as 1 and the same order of braid permutation as 4. 
Moreover, the closures of $b$ and $b'$ are both the trivial knot. 
We have $W^{\text{pur}}_0(b)=W_0(b^4)= \{ 4 \} = W^{\text{pur}}_0(b')$, $W^{\text{pur}}_1(b)=W_1(b^4)= \{ 6,6,6,6 \} =W^{\text{pur}}_1(b')$, and $W^{\text{pur}}_2(b)= \{ 8,8,8,8,10,10 \}$, $W^{\text{pur}}_2(b')= \{ 6,6,10,10,10,10 \}$. 
From $W^{\text{pur}}_2(b) \neq W^{\text{pur}}_2(b')$, we can conclude that $b$ and $b'$ are not conjugate by the contrapositive of Corollary \ref{cor-b-2l}.
\label{ex-np-W}
\end{example}

\subsection{Calculation of $(2,l)$ sets from crossing matrices} 

In this subsection, we observe that the purified $(2,l)$ sets can also be obtained from the crossing matrix and the braid permutation of $b$. \\

\noindent The braid invariants $w(b)$ and $W_1(b)$ of $b$ are obtained directly from $C(b)$ as follows. 

\medskip 
\begin{example}
The exponent sum $w(b)$ of an $n$-braid $b$ is derived from its crossing matrix $C(b)=[m_{i j}]$ as $w(b)= \sum_{i,j} m_{i j}=\mathbf{1}_{1 \times n} \ C(b) \ \mathbf{1}_{n \times 1}$, where $\mathbf{1}_{1 \times n}$ is the row vector $[ 1, 1, \dots , 1 ]$ and $\mathbf{1}_{n \times 1} = \mathbf{1}_{1 \times n}^T$. 
\label{ex-w-W1}
\end{example}
\medskip 

\begin{example}
Let $C(b)=[m_{i j}]$ be the crossing matrix of $b$. 
Since $w(f_i(2;b))=w(b)+w(s_i)$ and $w(s_i)= \sum_{k \neq i}m_{i k}+\sum_{k \neq i}m_{k i}$, we can calculate $W_1$ from $C(b)$. 
\end{example}
\medskip 

\noindent Thus, $w(b)$ and $W_1(b)$ are calculable from the crossing matrix $C(b)$. 
We can calculate $W_l(b)$ from $C(b)$ as follows.

\medskip 
\begin{example}
By Proposition \ref{prop-CM-2}, the crossing matrices of 2-cabled braids are obtained automatically from the crossing matrix and the braid permutation of the original braid.
Then the $(2,l)$ set $W_l$ is obtained by the calculation given in Example \ref{ex-w-W1}. 
Let $b$ be the pure 4-braid given in Example \ref{ex-w2}. 
Then we obtain $W_2(b)= \{ 2, 6, -2, 4, -6, -4 \}$ by the matrices
\begin{align*}
\begin{bmatrix}
0 & 0 & 0 & 0 & 1 & 0 \\
0 & 0 & 0 & 0 & 1 & 0 \\
0 & 0 & 0 & 0 & 1 & -1 \\
0 & 0 & 0 & 0 & 1 & -1 \\
1 & 1 & 1 & 1 & 0 & -1 \\
0 & 0 & -1 & -1 & -1 & 0 
\end{bmatrix}, \ 
\begin{bmatrix}
0 & 0 & 0 & 1 & 1 & 0 \\
0 & 0 & 0 & 1 & 1 & 0 \\
0 & 0 & 0 & 1 & 1 & -1 \\
1 & 1 & 1 & 0 & 0 & -1 \\
1 & 1 & 1 & 0 & 0 & -1 \\
0 & 0 & -1 & -1 & -1 & 0 
\end{bmatrix}, \ 
\begin{bmatrix}
0 & 0 & 0 & 1 & 0 & 0 \\
0 & 0 & 0 & 1 & 0 & 0 \\
0 & 0 & 0 & 1 & -1 & -1 \\
1 & 1 & 1 & 0 & -1 & -1 \\
0 & 0 & -1 & -1 & 0 & 0 \\
0 & 0 & -1 & -1 & 0 & 0 
\end{bmatrix}, \\
\begin{bmatrix}
0 & 0 & 0 & 1 & 1 & 0 \\
0 & 0 & 0 & 1 & 1 & -1 \\
0 & 0 & 0 & 1 & 1 & -1 \\
1 & 1 & 1 & 0 & 0 & -1 \\
1 & 1 & 1 & 0 & 0 & -1 \\
0 & -1 & -1 & -1 & -1 & 0 
\end{bmatrix}, \ 
\begin{bmatrix}
0 & 0 & 0 & 1 & 0 & 0 \\
0 & 0 & 0 & 1 & -1 & -1 \\
0 & 0 & 0 & 1 & -1 & -1 \\
1 & 1 & 1 & 0 & -1 & -1 \\
0 & -1 & -1 & -1 & 0 & 0 \\
0 & -1 & -1 & -1 & 0 & 0 
\end{bmatrix}, \ 
\begin{bmatrix}
0 & 0 & 1 & 1 & 0 & 0 \\
0 & 0 & 1 & 1 & -1 & -1 \\
1 & 1 & 0 & 0 & -1 & -1 \\
1 & 1 & 0 & 0 & -1 & -1 \\
0 & -1 & -1 & -1 & 0 & 0 \\
0 & -1 & -1 & -1 & 0 & 0 
\end{bmatrix}
\end{align*}
that are obtained from 
\begin{align*}
C(b)=
\begin{bmatrix}
0 & 0 & 1 & 0 \\
0 & 0 & 1 & -1 \\
1 & 1 & 0 & -1 \\
0 & -1 & -1 & 0 
\end{bmatrix}.
\end{align*}
\end{example}
\medskip

\noindent By Proposition \ref{prop-C-br}, the crossing matrices of powers of a braid are obtained automatically from the crossing matrix and the braid permutation of the original braid. 
The multiset $W^{\text{pur}}_l(b)$ is also calculable from $C(b)$ and the braid permutation of $b$. 

\medskip 
\begin{remark}
Corollaries \ref{cor-p-2l} and \ref{cor-b-2l} can also be proved by considering this calculation from the crossing matrix by Lemma \ref{lem-CM-p-perm} and Proposition \ref{prop-CM-pure-b}. 
\end{remark}

\section*{Acknowledgment}
This work was partially supported by the JSPS KAKENHI Grant Number JP21K03263.

\end{document}